\documentclass{amsart}
\usepackage{amsthm}
\usepackage{amssymb}
\usepackage{amsmath}
\usepackage{newpxtext,newpxmath}
  \makeatletter
    
    \@addtoreset{equation}{section}
  \makeatother

\title{The hard potential relativistic Boltzmann equation in the whole space}
\author{Koya Nishimura}
\date{} % delete this line to display the current date

\newtheorem{thm}{Theorem}[section]
\newtheorem{lem}[thm]{Lemma}
\newtheorem{prop}[thm]{Proposition}
\newtheorem{rem}[thm]{Remark}
\begin{document}

\maketitle
\begin{abstract}
For the hard potential relativistic Boltzmann equation in the whole space, we prove the global existence, uniqueness, and optimal time convergence rates to the relativistic Maxwellian. 
\end{abstract}
%\tableofcontents
\begin{center}{
\section{Introduction}
}\end{center}
The purpose of this paper is to handle the Cauchy problem for the relativistic Boltzmann equation with hard potentials (1.5). As stated in below, for hard potentials with some restrictions, Yong Wang [6] proved the global existence and uniqueness by using the excess conservations of mass, momentum and energy (1.7) and the excess entropy inequality (1.8). Moreover, it was also shown that the solutions in the periodic domain $\mathbb{T}^3_x$ decay in time at the exponential rate. In particular, this result works even for initial perturbations large in the $L^{\infty}_{x,p}$ space. 
Robert M. Strain and Keya Zhu [5] proved that the small solutions in the whole space $\mathbb{R}^3_x$ decay at the optimal algebraic rate of $(1+t)^{-3/4}$ for soft potentials. To handle the whole space case, the $L^\infty_p (L^2_x \cap L^\infty_x)$ space is used instead of the $L^\infty_{x,p}$ space. Our approach is based on their works. In this paper, for the hard potentials (1.5) we prove the global existence, uniqueness and the optimal time decay rate for the solution. Also, our local existence result (Theorem 3.1) is valid in a particular case of hard potentials.%Robert M. Strain and Keya Zhu [] proved that the small solutions in the whole space decay at the optimal algebraic rate of $(1+t)^{-3/4}$ () for soft potentials (). Combining the two approaches used in [] and [], we will handle the full hard potential case (\ref{9}) for large initial data in $L^\infty_{x,p}$ . 

The dynamics of particles whose speed is comparable to the speed of light is described by the relativistic Boltzmann equation:
\begin{equation}
\label{1}
\partial_t F+\hat{p} \cdot \nabla_x F=Q(F,F), \quad F(0,x,p)=F_0 (x,p).
\end{equation}
The solution, $F=F(t,x,p)\geq0,$ is a distribution function for the particles at time $t \in \mathbb{R}_{+},$ position $x \in \mathbb{R}^3,$ and momentum $p \in \mathbb{R}^3 .$ The normalized velocity $\hat{p}$ is defined by 
\begin{equation}
\label{2}
\hat{p}={p \over p^0}={{p} \over {\sqrt{1+|p|^2}}}. \nonumber
\end{equation}
The collision operator, $Q(F,F),$ is defined as
\begin{equation}
\label{3}
Q(F,G)=\int_{\mathbb{R}^3}dq \int_{\mathbb{S}^2}d\omega \,v_{\phi}\sigma(g,\theta)\left[F(p')G(q')-F(p)G(q)\right].
\end{equation}
Here 
\begin{equation}
\label{4}
g=g(p,q)=\sqrt{2(p^0 q^0 -p\cdot q-1)},\quad s=g^2+4,\nonumber
\end{equation}
\begin{equation}
\label{5}
v_{\phi}=v_{\phi}(p,q)=\sqrt{\left|{{p \over p^0}-{q \over q^0}}\right|^2 -\left|{{p \over p^0} \times{q \over q^0}}\right|^2}={{g \sqrt{s}} \over{p^0 q^0}},\nonumber
\end{equation}
\begin{equation}
\label{6}
p'={{p+q} \over 2}+{g \over 2} \left[\omega+\left({{p^0+q^0}\over{\sqrt{s}}}-1\right)(p+q){{(p+q) \cdot \omega} \over{|p+q|^2}}\right],\nonumber
\end{equation}
\begin{equation}
\label{7}
q'={{p+q} \over 2}-{g \over 2} \left[\omega+\left({{p^0+q^0}\over{\sqrt{s}}}-1\right)(p+q){{(p+q) \cdot \omega} \over{|p+q|^2}}\right].\nonumber
\end{equation}
Above $p'$ and $q'$ are the post-collisional momenta for the pre-collisional momenta $p$ and $q,$ respectively, which satisfy
\begin{equation}
\label{8}
\begin{gathered}
p^0 +q^0 =p'^0 +q'^0, \\
         p+q=p'+q'.
\end{gathered}
\end{equation}
The angle $\theta$ satisfies $\cos \theta=k \cdot \omega$ with $k=k(p,q)$ and $|k|=1.$ We assume that the scattering kernel, $\sigma=\sigma(g,\theta),$ satisfies the following hard potential assumption :
\begin{equation}
\label{9}
\begin{gathered}
\left(g \over \sqrt{s} \right) g^a \sigma_0 (\theta) \lesssim \sigma(g,\theta) \lesssim (g^a +g^{-b})\sigma_0 (\theta).
\end{gathered}
\end{equation}
Here $0 \leq \sigma_0(\theta) \lesssim \sin^{\gamma} (\theta),$ $\gamma>-2,$ $0\leq a \leq 2+\gamma$ and $0\leq b<\min\{4,4+\gamma\}.$ Additionally $\sigma_0 (\theta)$ should be non-zero on a set of positive measure. For the definition of the vector $k$ and explanations about the above assumption, we refer to [3] and the references therein.

We define a normalized relativistic Maxwellian by
\begin{equation}
\label{10}
J(p)={1 \over {4 \pi}} e^{-p^0}, \nonumber
\end{equation}
which satisfies (\ref{1}) since $Q(J,J)=0$ by (\ref{8}). As in [1], for solutions to (\ref{1}), we introduce the excess conservations of mass, momentum and energy,
\begin{equation}
\label{11}
\begin{gathered}
\int_{\mathbb{R}^3 \times \mathbb{R}^3}dx\,dp\, \{F(t,x,p)-J(p)\}=\int_{\mathbb{R}^3 \times \mathbb{R}^3}dx\,dp\, \{F_0(x,p)-J(p)\}\equiv M_0, \\
\int_{\mathbb{R}^3\times \mathbb{R}^3}dx\,dp\, \{F(t,x,p)-J(p)\}p=\int_{\mathbb{R}^3 \times \mathbb{R}^3 }dx\,dp\, \{F_0(x,p)-J(p)\}p\equiv J_0,\\
\int_{\mathbb{R}^3 \times \mathbb{R}^3}dx\,dp\, \{F(t,x,p)-J(p)\}p^0=\int_{\mathbb{R}^3\times \mathbb{R}^3}dx\,dp\, \{F_0(x,p)-J(p)\}p^0\equiv E_0, \\
\end{gathered}
\end{equation}
as well as the excess entropy inequality,
\begin{equation}
\label{12}
\begin{gathered}
\hspace{-25mm}\int_{\mathbb{R}^3 \times \mathbb{R}^3}dx\,dp\, \{F(t,x,p)\ln F(t,x,p)-J(p)\ln J(p)\}\\ \hspace{15mm} \leq \int_{\mathbb{R}^3\times \mathbb{R}^3}dx\,dp\, \{F_0(x,p)\ln F_0 (x,p)-J\ln J\}\equiv H_0.
\end{gathered}
\end{equation}

Next we define the notation and the perturbation to the relativistic Maxwellian.\\

\subsection{Notation and perturbation equation}
We define the Lebesgue spaces for scalar functions $g$ by $$\left|\left|g\right|\right|_{L^{r_1}_t L^{r_2}_p L^{r_3}_x}=\left[\int_0^\infty \left\{ \int_{\mathbb{R}^3} \left(\int_{\mathbb{R}^3} \left|g(t,x,p)\right|^{r_3} dx \right)^{r_2 \over r_3} dp \right\}^{{r_1} \over r_2} dt\right]^{1 \over r_1},\quad r_1,\,r_2,\,r_3 \in[1,\infty),$$
with the standard modifications when $r_1,\,r_2$ or $r_3=\infty.$ Similarly for $||\cdot||_{L^{r_2}_p L^{r_3}_x}.$ 

For $l \in \mathbb{R},$ we define a weight function in $p$ by $$w_l (p)=(p^0)^l=\sqrt{1+|p|^2}^{l}.$$ We also define a weight function in $t$ by $\varpi_r (t)=(1+t)^{\sigma_r}$ where 
\begin{equation}
\sigma_r={3 \over 2}\left({1 \over r} -{1 \over 2}\right),\quad 1 \leq r \leq 2.
\end{equation}
Throughout this paper, $C$ denotes some positive (generally large) constant and $c$ denotes some positive (generally small) constant, where both $C$ and $c$ may depend on $\gamma,\,a,\,b,\,l$ and take different values in different places. Furthermore, we use the notation $A \lesssim B$ to mean that $A \leq CB$ for a constant $C>0$ such as above. 

We define the standard perturbation $f(t,x,p)$ to the relativistic Maxwellian $J$ as $$F=J+\sqrt{J}f.$$
The relativistic Boltzmann equation for the perturbation $f$ takes the form
\begin{equation}
\partial_t f +\hat{p} \cdot \nabla_x f +\nu (p) f-K(f)=\Gamma(f,f), \quad f (0,x,p)=f_0 (x,p). 
\end{equation}
Above the multiplication operator is
\begin{equation}
\nu(p)=\int_{\mathbb{R}^3} dq \int_{\mathbb{S}^2}d\omega\,v_{\phi}\,\sigma(g,\theta)J(q). \nonumber
\end{equation}

\noindent The integral operator is defined by
\begin{equation}
\begin{gathered}
K(h)=\int_{\mathbb{R}^3} dq \int_{\mathbb{S}^2}d\omega\,v_{\phi}\,\sigma(g,\theta)\sqrt{J(q)}\left\{\sqrt{J(q')}h(p')+\sqrt{J(p')}h(q')\right\}\\
-\int_{\mathbb{R}^3} dq \int_{\mathbb{S}^2}d\omega\,v_{\phi}\,\sigma(g,\theta)\sqrt{J(q)J(p)}h(q). \nonumber
\end{gathered}
\end{equation}
The nonlinear part of the collision operator is defined as
\begin{equation}
\begin{gathered}
\hspace{-60mm}\Gamma(h_1,h_2)=J^{-{1 \over 2}} Q(\sqrt{J}h_1,\sqrt{J}h_2)\\ \hspace{18mm}=\int_{\mathbb{R}^3} dq \int_{\mathbb{S}^2}d\omega\,v_{\phi}\,\sigma(g,\theta)\sqrt{J(q)}\left[h_1(p')h_2(q')-h_1(p)h_2(q)\right]. \nonumber
\end{gathered}
\end{equation}
It is easily verified that the mild form of (1.9) is given by
\begin{equation}
\begin{gathered}
\hspace{-40mm}f(t,x,p)=e^{-\nu(p)t}f_0 \left(x-\hat{p}t,p\right)\\ \hspace{20mm}+\int_0^t e^{-\nu(p)(t-s)} \,K \left(f\right) \left(s,x-\hat{p}(t-s),p\right)\,ds\\ \hspace{20mm}+\int_0^t e^{-\nu(p)(t-s)}\,\Gamma \left(f,f\right) \left(s,x-\hat{p}(t-s),p\right)\,ds.
\end{gathered}
\end{equation}

Our main results are as follows.
\subsection{Main results}
\begin{thm}Let $l>14+{a \over 2},\,M \geq 1,\,\left|\left|w_l f_0\right|\right|_{L^\infty_{p} \left(L^2_x \cap L^\infty_x \right)}  \leq M< \infty,$ and $F_0 (x,p)=J(p)+\sqrt{J(p)}f_0 (x,p).$ There exist $\epsilon>0$ and $L>0$ depending on $a,\,b,\,\gamma,\,l,\,M$ such that if
\begin{equation}
\sup_{t\geq 0,\,x \in \mathbb{R}^3} \int_{|p| \leq L}\left|f_0 \left(x-\hat{p}t,p\right)\right| \,dp+\left|M_0\right|+\left|E_0\right|+\left|H_0\right| \leq \epsilon,
\end{equation}
then there is a unique global solution (1.9), $f(t,x,p),$ to the relativistic Boltzmann equation (1.8) satisfying
\begin{equation}
\sup_{t \geq 0} \left|\left|w_l f(t)\right|\right|_{L^\infty_{p} \left(L^2_x \cap L^\infty_x \right)}  \leq CM^2. \nonumber
\end{equation}
Moreover, (1.5) and (1.6) hold. The solution $f(t,x,p)$ is continuous if it is so initially. The positivity $F(t,x,p)=J(p)+\sqrt{J(p)} f(t,x,p)\geq0$ also holds.
\end{thm}
\begin{rem}For the case $a \in [0,2] \cap [0,2+\gamma)$ the smallness of 
\begin{equation}
\sup_{0 \leq t \leq T^\star,\,x \in \mathbb{R}^3} \int_{|p| \leq L}\left|f_0 \left(x-\hat{p}t,p\right)\right| \,d p
\end{equation}
 is not needed, where $[0,T^\star]$ is the time interval of the local solutions (constructed in Theorem 1.5 or Theorem 3.1 below). The result extends Theorem 1.6 below.

For the existence, we can remove the assumption on the $L^\infty_p L^2_x$ norm, but the norm will be needed in the proof of Theorem 1.3 below. We can also prove Theorem 1.1 for the periodic domain.
\end{rem}
\begin{thm}
Choose $r \in [1,6/5)$ where $\sigma_r$ is given by (1.9). Let $f(t,x,p)$ be the solution constructed in Theorem 1.1. For any fixed $r \in [1,6/5),$ if $\left|\left| f_0\right|\right|_{L^2_{p}L^r_x} <\infty$. Then 
\begin{equation}
\begin{gathered}
\left|\left|w_l f(t)\right|\right|_{L^\infty_{p} \left(L^2_x \cap L^\infty_x \right)}  \leq C (1+t)^{-\sigma_r} \left\{\left|\left|w_l f_0\right|\right|_{L^\infty_{p} \left(L^2_x \cap L^\infty_x \right)} +\left|\left| f_0\right|\right|_{L^2_{p}L^r_x} \right\}. \nonumber
\end{gathered}
\end{equation}
\end{thm}
\begin{rem}When $r=1$, the decay rate is $(1+t)^{-{3 \over 4}}$ and this is optimal (see [7]).
\end{rem}
\subsection{Previous results}
In this subsection, we consider both $\mathbb{T}^3_x$ and $\mathbb{R}^3_x$ case. Let $x \in \Omega=\mathbb{R}^3$ or $\mathbb{T}^3$. We recall some of the results in \cite{r6}.
\begin{thm}[Local existence]\cite{r6}. We assume that $-2<\gamma,\,a\in [0,2] \cap [0,2+\gamma),\,b \in [0,\min\{4,4+\gamma\}).$ Let $\beta>14,\,F_0 (x,p)=J+\sqrt{J} f_0 (x,p) \geq 0$ and $\left|\left|w_\beta f_0 \right|\right|_{L^\infty_{x,p}} < \infty,$ then there exists a positive time 
\begin{equation}
T^\star=\left[ \widetilde{C} \left(1+\left|\left|w_\beta f_0\right|\right|_{L^\infty_{x,p}} \right) \right]^{-1}>0, \nonumber
\end{equation}
such that the relativistic Boltzmann equation (1.1) has a unique mild solution $F(t,x,p)=J(p)+\sqrt{J(p)}f(t,x,p) \geq 0$ on the time interval $t \in [0, T^\star]$ and satisfies 
\begin{equation}
\left|\left|w_\beta f (t) \right|\right|_{L^\infty_{x,p}} \leq 2 \left|\left|w_\beta f_0\right|\right|_{L^\infty_{x,p}},\quad \mbox{for } 0 \leq t \leq T^\star, \nonumber
\end{equation}
where the positive constant $\widetilde{C} \geq 1$ depends only on $a,\,b,\,\gamma,\,\beta.$ In addition, the conservations of mass, momentum, and energy (\ref{11}) as well as the additional entropy inequality (\ref{12}) hold. Furthermore, if the initial data $f_0$ is continuous, then the solution $f(t,x,p)$ is continuous in $[0,T^\star] \times \Omega \times \mathbb{R}^3.$
\end{thm}

\begin{thm}[Global existence]\cite{r6}. In addition to the above assumptions of Theorem 1.5, we restrict $\gamma>-{4 \over 3},\,a \in [0,2] \cap [0,\min\{2+\gamma,4+3\gamma \}),\,b \in [0,2).$ If $\left|\left|w_\beta f_0 \right|\right|_{L^\infty_{x,p}} \leq \overline{M},$ there is a small constant $\epsilon_0>0$ depending on $a,\,b,\,\gamma,\, \beta,\,\overline{M}$ such that if 
\begin{equation}
\left|M_0\right|+\left|E_0\right|+\left|H_0\right|+\left|\left|f_0 \right|\right|_{L^1_x L^\infty_p} \leq \epsilon_0,
\end{equation}
the relativistic Boltzmann equation (1.1) has a global unique mild solution $F(t,x,p)=J(p)+\sqrt{J(p)} f(t,x,p)\geq0 $ satisfying (\ref{11}), (\ref{12}) and
\begin{equation}
\sup_{t \geq 0} \left|\left|w_\beta f(t) \right|\right|_{L^\infty_{x,p}} \leq \widetilde{C}_1\overline{M}^2, \nonumber
\end{equation}
where the positive constant $\widetilde{C}_1$ depends only on $a,\,b,\,\gamma,\,\beta.$ Moreover, if the initial data $f_0$ is continuous, then the solution $f(t,x,p)$ is continuous in $[0,\infty) \times \Omega \times \mathbb{R}^3.$
\label{th2}
\end{thm}
\begin{rem}
In (1.12) the norm $\left|\left|f_0\right|\right|_{L^1_x L^\infty_p}$ can be replaced by 
\begin{equation}
\sup_{t \geq T^\star,\,x \in \mathbb{R}^3}\int_{\mathbb{R}^3} e^{-\nu(p)t} \left|f_0 \left(x-\hat{p} t,p\right) \right|\, dp.
\end{equation}
\end{rem}
\begin{thm}[Time decay]\cite{r6}. Let $\Omega= \mathbb{T}^3.$ Assume that $[M_0,J_0,E_0]=[0,0,0],$ and $\epsilon_0>0$ sufficiently small, then there exists a positive constant $\lambda_0 >0$ such that the solution $f(t,x,p)$ obtained in Theorem \ref{th2} satisfies 
\begin{equation}
\left|\left|w_\beta f(t) \right|\right|_{L^\infty_{x,p}} \leq \widetilde{C}_2 e^{-\lambda_0 t},\nonumber
\end{equation}
where $\widetilde{C}_2>0$ is a positive constant depending only $a,\,b,\,\gamma,\,\beta$ and $\overline{M}.$
\end{thm}

We also refer to [3,5,7] for recent works about the small initial perturbation. These results use an energy method for solutions to the linearized relativistic Boltzmann equation. For a historical discussion, see [5].

\subsection{Outline of the paper} The rest of this article is organized as follows. First, in Section 2, we recall some basic estimates from [5] and [6]. In Section 3, we construct the local solution. The global existence (Theorem 1.1) is obtained in Section 4. Lastly, in Section 5 we briefly explain the linear $L^\infty_{x,p}$ time decay for solutions to the linearized equation. By using the results, we prove the non-linear time decay rates (Theorem 1.3).
\section{Basic estimates}
\noindent As in \cite{r5}, given a small $\epsilon>0,$ we choose a smooth cut-off function $\chi=\chi(g)$ satisfying
\begin{equation*}
 \chi(g)=\begin{cases}
      1& \text{if $g \geq 2\epsilon$}, \\
      0& \text{if $g \leq \epsilon$}.
    \end{cases} 
\end{equation*}  
We split the integral operator $K$ into
$$K=K^{\chi}+K^{1-\chi},$$
where
\begin{equation}
\begin{gathered}
K^{1-\chi}(h)(p)=\int_{\mathbb{R}^3} d q \int_{\mathbb{S}^2}d\omega\,[1-\chi(g)]\,v_{\phi}\,\sigma(g,\theta)\sqrt{J(q)}\sqrt{J(q')}h(p')\\+\int_{\mathbb{R}^3} d q \int_{\mathbb{S}^2}d\omega\,[1-\chi(g)]\,v_{\phi}\,\sigma(g,\theta)\sqrt{J(q)}\sqrt{J(p')}h(q')\\
-\int_{\mathbb{R}^3} d q \int_{\mathbb{S}^2}d\omega\,[1-\chi(g)] \,v_{\phi}\,\sigma(g,\theta)\sqrt{J(q)J(p)}h(q). \nonumber
\end{gathered}
\end{equation}
The remaining operator can be expressed by $$K^\chi(h)(p)=\int_{\mathbb{R}^3} dq \,k^\chi (p,q)\, h(q).$$
We can also express $K$ as
$$K (h)(p)=\int_{\mathbb{R}^3} dq\, k(p,q) h(q)=\int_{\mathbb{R}^3} dq\, \left\{k^a (p,q)+k^b (p,q)\right\} h(q).$$
For the formulas of the kernels $k^\chi,\,k^a$ and $k^b$, we refer to Section 2 of [6] for instance.

We can obtain the following estimates under the hard potential assumption (1.5). 
\begin{lem}\cite{r6}. We have
\begin{equation*}
  \nu(p) \approx
      \left(p^0\right)^{a \over 2}. 
      \end{equation*}
\end{lem}
%lem2.2
\begin{lem}\cite{r6}. We have
\begin{equation}
k^a (p,q) \lesssim \begin{cases}
\left(p^0 q^0\right)^{{a-\gamma-2} \over 4} e^{-c|p-q|} & \text{for $a \geq \gamma \geq 1$,}\\
\left(p^0 q^0\right)^{{a-\gamma-2} \over 4} g^{\gamma-1} (p,q) e^{-c|p-q|} & \text{for $a \geq \gamma \geq 0, \quad \gamma<1$,}\\
\left(p^0 q^0\right)^{-{1 \over 2}+{{\zeta_1} \over 4}} e^{-c|p-q|} & \text{for $1 \leq a <\gamma$,}\\
\left(p^0 q^0\right)^{-{1 \over 2}+{{\zeta_1} \over 4}} g^{a-1} (p,q)e^{-c|p-q|} & \text{for $a<\gamma, \quad a <1$,}\\
\left(p^0 q^0\right)^{{a+|\gamma|-2} \over 4} g^{-|\gamma|-1}(p,q) e^{-c|p-q|} & \text{for $-2<\gamma<0,\quad a\leq 2+\gamma$,}\\
\end{cases}
\end{equation}
and,
\begin{equation}
k^b (p,q) \lesssim \begin{cases}
\left(p^0 q^0\right)^{-{1 \over 2}-{\zeta_2 \over 4}} g^{-b-1} (p,q)e^{-c|p-q|} & \text{for $\gamma \geq 0,\quad b<2$},\\
\left(p^0 q^0\right)^{ {|\gamma|-b-2} \over 4} g^{-b-1} (p,q) e^{-c|p-q|} & \text{for $-b<\gamma < 0,\quad |\gamma|<b<2$},\\
\left(p^0 q^0\right)^{{|\gamma|-b-2} \over 4} g^{-|\gamma|-1} (p,q) e^{-c|p-q|} & \text{for $-2< \gamma< 0,\quad |\gamma|\geq b$},\\
\end{cases}
\end{equation}
where $\zeta_1=\max\{-2,a-\gamma\}$, $\zeta_2=\min\{2,b+\gamma\}.$
\end{lem}
\begin{lem}We have
\begin{equation*}
  \left| k^\chi (p,q) \right| \leq C_\chi
      \left\{\left(p^0 q^0\right)^{-\zeta_a+{a \over 4}} +\left(p^0 q^0\right)^{-\zeta_b -{b \over 4}}\right\}e^{-c|p-q|} \leq C_\chi e^{-c|p-q|}.
      \end{equation*}
where $\zeta_a=\min\{2-|\gamma|,4+a\}/4>0$ and $\zeta_b=\min\{2-|\gamma|,4-b\}>0.$
\end{lem}
\begin{proof}
The lemma is easily obtained from Lemma 2.2.
\end{proof}

\begin{lem}Fix any $l\geq0$ and any $j>0.$ Given any small $\eta>0,$ which depends upon $\chi,$ we have
$$\left|w_l (p) K^{1-\chi}(h)(p)\right| \leq \eta e^{-cp^0} ||w_{-j} h||_{L^{\infty}_q}.$$
\end{lem}
\begin{proof}
We can prove the lemma similarly for Lemma 4.6 in \cite{r3}.
\end{proof}

We will use the following estimates to construct global solutions.
%lem2.5
\begin{lem}\cite{r6}. Let F be a solution to the relativistic Boltzmann equation (1.1). We have 
\begin{equation}
\begin{gathered}
\hspace{-15mm}\int_{\mathbb{R}^3 \times \mathbb{R}^3} {\left|F(t,x,p)-J(p)\right|^2 \over 4J(p)}\, \chi_{|F(t,x,p)-J(p)| \leq J(p)}\, d x \,d p\\+\int_{\mathbb{R}^3 \times \mathbb{R}^3} {\left|F(t,x,p)-J(p)\right| \over 4}\,\chi_{|F(t,x,p)-J(p)|\geq J(p)} \,d x \,d p\\ \lesssim |M_0|+|E_0|+|H_0|,
\end{gathered}
\end{equation}
where $\chi_{|F(t,x,p)-J(p)| \leq J(p)}$ is the characteristic function of $\{(t,x,p): |F(t,x,p)-J(p)| \leq J(p)\}$ and $\chi_{|F(t,x,p)-J(p)|\geq J(p)}=1-\chi_{|F(t,x,p)-J(p)| \leq J(p)}.$  
\label{lemg1}
\end{lem}
%lem2.6
\begin{lem}\cite{r6}. Let $l \geq1$ and $1<d<\min\left\{{9 \over 8}, {2 \over {\max\{-\gamma,1\}}},{3 \over \max\{b-1,1\}}\right\}.$ 
We have
\begin{equation}
w_l (p) \left|\left|\Gamma (f,f) \right|\right|_{L^\infty_x} (s,p) \leq C \nu (p) \left|\left|w_l f(s)\right|\right|^{{9d+1} \over 5d}_{L^{\infty}_{x,p}} \sup_{y \in \mathbb{R}^3} \left\{\int_{\mathbb{R}^3} |f(s,y,q)|\,dq\right\}^{{d-1} \over 5d},
\end{equation}
\begin{equation}
w_l (p) \left|\left|\Gamma (f,f) \right|\right|_{L^1_x} (s,p) \leq C \nu (p) \left|\left|w_l f(s)\right|\right|^{{9d+1} \over 5d}_{L^{\infty}_p L^2_x}\left\{\int_{\mathbb{R}^3} \left|\left|f\right|\right|_{L^2_x} (s,q)\,dq\right\}^{{d-1} \over 5d},
\end{equation}
where the constant $C>0$ depends only on $a,\,b,\,\gamma,\,d.$
\label{lemg2}
\end{lem}
\begin{proof}
For (2.4) we refer to Lemma 4.1 in [6]. Similarly, we can get (2.5) by using Cauchy-Schwartz inequality.
\end{proof}
\begin{center}{
\section{Local existence}
}\end{center}
We can handle the case $a\geq2,\,\gamma\geq0.$ 
\begin{thm}
\label{ls}
Let $l >14+a/2,$ $\left|\left|w_{l} f_0 \right|\right|_{L^\infty_p (L^2_x \cap L^{\infty}_x)}\leq \widetilde{M}/2<\infty,$ and $F_0=J+\sqrt{J} f_0 (x,p) \geq 0.$ There exist $T^\star>0,$ $B>0,$ and $\epsilon>0$ such that if $T^\star \lesssim \widetilde{M}^{-1},$ $B\gtrsim \widetilde{M},$ and 
\begin{equation}
\sup_{0 \leq t \leq T^\star,\, x \in \mathbb{R}^3}\int_{|p|<B} \left|f_0 \left(x-\hat{p}t,p\right) \right|dp\leq \epsilon,
\end{equation}
then there is a unique local solution (1.9), f(t,x,p), to $(1.8)$ in $[0,T^\star] \times \mathbb{R}^3 \times \mathbb{R}^3$ satisfying $$\sup_{0\leq t \leq T^\star} \left|\left|w_{l} f(t)\right|\right|_{L^\infty_p (L^2_x \cap L^\infty_x)} \leq \widetilde{M}.$$
Moreover, if $M_0,\,E_0,$ and $H_0$ are finite, then (1.5) and (1.6) hold. The solution $f(t,x,p)$ is continuous if it is so initially. The positivity $F=J+\sqrt{J} f\geq0$ also holds. 
\end{thm}
\begin{proof}
We use the following iterating sequence $(n \geq0)$:
\begin{equation}
\begin{gathered}
\hspace{-35mm}f^{n+1} (t,x,p)=f_0 \left(y,p\right) e^{-\int_0^t B^n \left(s_1,y+\hat{p} s_1,p\right)\,d s_1}\\
\hspace{0mm}+\int_0^t e^{-\int_s^t B^n \left(s_1,y+\hat{p}s_1,p\right)\,d s_1}K f^n \left(s,y+\hat{p}s,p \right)\,ds \\ \hspace{0mm}+\int_0^t e^{-\int_s^t B^n \left(s_1,y+\hat{p}s_1,p \right)\,d s_1}\Gamma_{gain} \left(f^n,f^n \right) \left(s,y+\hat{p}s,p \right)\,ds 
\end{gathered}
\end{equation}
with $f^0=0$ and $f^{n+1}|_{t=0} =f_0.$
Here we have used the notation $y=x-\hat{p}t,$ 
\begin{eqnarray*}
B^n (s_1,x,p)&=&\int_{\mathbb{R}^3} dq\int_{\mathbb{S}^2} d\omega\,v_{\phi} \sigma (g,\theta)F^n (s_1,x,q)\\ &=&\int_{\mathbb{R}^3} dq\int_ {\mathbb{S}^2}d\omega\,v_{\phi} \sigma (g,\theta)\left[J(q)+\sqrt{J(q)}f^n (s_1,x,q)\right],
\end{eqnarray*}
and
\begin{equation}
\Gamma_{gain}(f^n,f^n)=\int_{\mathbb{R}^3} dq\int_{\mathbb{S}^2} d\omega\,v_{\phi} \sigma(g,\theta) \sqrt{J(q)}f^n (p')f^n (q').
\end{equation}
By an induction and $F_0 \geq 0$, we can easily obtain $F^n=J+\sqrt{J} f^n \geq 0$ for all $n$, and then $B^n (s_1,x,p)\geq 0$.

Assume that for $k=n-1,n,$ $$\sup_{0 \leq t \leq T^\star} \left|\left|w_l f^k (t)\right|\right|_{L^\infty_p (L^2_x \cap L^\infty_x)}\leq \widetilde{M}.$$ 
We show that for $k=n+1$ the above holds. 
For the case $a \geq 0,\,\gamma \geq 0$, we will need to control $-B^n$. Clearly, for any $B>0$ we have
\begin{eqnarray}
\begin{gathered}
\hspace{-50mm}-B^n (s_1,y+\hat{p}s_1,p) \\ \hspace{20mm}\lesssim -\nu (p)+{\left(p^0\right)}^{a \over 2} \int_{|q|<B} \left(q^0\right)^{a \over 2} \sqrt{J(q)}|f^n (s_1,y+\hat{p} s_1,q)| \,dq\\ \hspace{20mm}+\widetilde{M} \left(p^0\right)^{a \over 2}\int_{|q|>B} \left(q^0\right)^{a \over 2}\sqrt{J(q)} \,dq.
\end{gathered}
\end{eqnarray}
Above the third term on the r.h.s. is bounded by $\nu(p) \over 4$ if $B$ is large enough. For the second term we use (3.2) to get
\begin{equation}
\begin{gathered}
\hspace{-0mm}\int_{|q|<B}  \left|f^n \left(s_1,y+\hat{p} s_1,q \right)\right| \,dq \leq \int_{|q|<B}\left| f_0 \left(y+ \hat{p}s_1-\hat{q}s_1,q \right) \right| \,dq\\+\int_0^{s_1} \int_{|q|<B} \left|K\left(\left|\left|f^{n-1}\right|\right|_{L^\infty_x}\right)\left(s_2,q \right) \right|\,dq\,d s_2\\+\int_0^{s_1} \int_{|q|<B} \left|\Gamma_{gain}\left(\left|\left|f^{n-1} \right|\right|_{L^\infty_x},\left|\left|f^{n-1} \right|\right|_{L^\infty_x} \right)\left(s_2,q \right) \right| \,dq\,d s_2
\end{gathered}
\end{equation}
We denote the second term and third term on the r.h.s. of (3.5) by $H_1$ and $H_2,$ respectively. For $H_1$ we split $K=K^\chi +K^{1-\chi}$. Then for $l>3,$ from Lemmas 2.3 and 2.4, we have
\begin{equation}
H_1 \leq C\widetilde{M} s_1 \int_{\mathbb{R}^3} w_{-l}(q_1) \int_{\mathbb{R}^3} k^{\chi}(q,q_1)\,dq \,dq_1+C\widetilde{M} s_1 \int_{\mathbb{R}^3} e^{-cq_0}\,dq \leq C\widetilde{M}T^\star. \nonumber
\end{equation}
Next we estimate $H_2$. Notice that
\begin{equation}
H_2\leq \widetilde{M}^2 s_1  \int_{|q|<B} \int_{\mathbb{R}^3} \int_{\mathbb{S}^2} v_{\phi} (q,q_1)\sigma(g(q,q_1),\theta)\sqrt{J(q_1)}w_{-l}(q')w_{-l}(q'_1)\,d\omega\,dq_1\,dq,  \nonumber
\end{equation}
and using $q'^0+{q'_1}^0=q^0+q_1^0$ that
\begin{equation}
w_l (q')w_l (q_1')=\left({q'}^0 {q'_1}^0\right)^l \gtrsim(q'^0+{q'_1}^0)^l\geq (q^0)^l=w_l(q). \nonumber
\end{equation}
Then for $l > 3+a/2$ we have 
\begin{equation}
\begin{gathered}
H_2 \leq C\widetilde{M}^2 s_1  \int_{|q|<B} w_{-l}(q)\left\{\int_{\mathbb{R}^3} \int_{\mathbb{S}^2} v_{\phi} (q,q_1)\sigma(g(q,q_1),\theta)\sqrt{J(q_1)}\,d\omega\,dq_1\right\}\,dq\\ \leq C\widetilde{M}^2 s_1 \int_{\mathbb{R}^3}w_{a/2-l}(q)\,dq\leq C\widetilde{M}^2 T^\star. \nonumber
\end{gathered}
\end{equation}
If $T^\star$ is sufficiently small, then we obtain $$\sup_{0\leq s_1 \leq T^\star} \int_{|q|<B} |f^n (s_1,y+\hat{p} s_1,q)| \leq 2 \epsilon,$$
and then $-B^ n(s_1,y+\hat{p}s_1,p) \leq -c\nu(p)$ when $0 \leq s_1 \leq T^\star$ and $\epsilon$ is small. 

We apply the last inequality to the third term of (3.2). Then
\begin{equation}
\begin{gathered}
w_{l}(p)\left|f^{n+1} (t,x,p)\right|\leq w_{l}(p)| f_0 (y,p)|+\int_0^t w_l (p)\left|K(f^n)(s,y+\hat{p}s,p)\right|\,ds\\+\int_0^t w_l (p) e^{-c\nu(p){(t-s)}}\left|\Gamma_{gain} (f^n,f^n)(s,y+\hat{p}s,p)\right|\,ds.
\end{gathered}
\end{equation}
Now we estimate the $L^\infty_p \left(L^2_x \cap L^\infty_x\right)$ norm of $w_l f^{n+1}$. First, as in the estimate for $H_1$ we have
$$\left|\left|\int_0^t w_l (p) K \left(f^n \right) (s,y+\hat{p}s,p)\,ds \right|\right|_{L^2_x \cap L^\infty_x} \leq C T^\star \left|\left|w_l f^n (s)\right|\right|_{L^\infty_q (L^2_x \cap L^\infty_x)}.$$
For the last term of (3.6), for $l >14+a/2$ we can obtain
\begin{equation}
\begin{gathered}
\hspace{-45mm}w_l (p)\left|\left|\Gamma_{gain} (f^n,f^n)\right|\right|_{L^2_x \cap L^{\infty}_x} (s,p)\\ \hspace{25mm}\lesssim \left\{\left(p^0\right)^{- {b \over 2}}+\int_0^t e^{-c\nu(p)(t-s)}A(p)\,ds\right\}\left|\left|w_l f^n (s)\right|\right|^2_{L^\infty_q (L^2_x \cap L^\infty_x)},
\end{gathered}
\end{equation}
where
\begin{equation}
\label{la1}
  A(p)=\begin{cases}
      \left(p^0\right)^{-1+{a \over 2}}& \text{if $\gamma \geq 0$}, \\
      \left(p^0\right)^{-1+{a \over 2}}+\left(p^0\right)^{{1 \over 2}(a+|\gamma|-2)+O(\delta)}& \text{if $-2<\gamma<0,\quad \forall \delta >0$.}
    \end{cases}
\end{equation}
For simplicity, we omit the proof of (3.7) (we refer to the proof of Theorem 3.1 in [6]). 

Let $N \gg 1.$ For $\gamma \geq0,$ we have
\begin{equation*}
\begin{gathered}
\hspace{-45mm}\int_0^t e^{-c\nu(p)(t-s)} A(p)\left\{\chi_{p^0\leq N}+\chi_{p^0\geq N}\right\}\,ds\\ \hspace{10mm}\lesssim N^{{a \over 2}}\int_0^t e^{cN^{a \over 2}s} ds+N^{-1}\int_0^t e^{-c\nu(p)(t-s)} \nu(p) ds\hspace{0mm}\lesssim N^{{a \over 2}}t+N^{-1},
\end{gathered}
\end{equation*}
where $\chi_{p^0 \leq N}$ is the characteristic function of $\{p:p^0\leq N\}$ and $\chi_{p^0\geq N}=1-\chi_{p^0 \leq N}.$
For $\gamma<0,$ we choose $\delta=\delta(\gamma)$ small so that $|\gamma|-2+O(\delta)<0.$ The same estimate as the above case yields 
$$\int_0^t e^{-c\nu(p)(t-s)} A(p) \,ds \lesssim N^{{a \over 2}}t+N^{-1}+N^{|\gamma|-2+O(\delta)}.$$

\noindent{Hence }
\begin{eqnarray*}
\begin{gathered}
\hspace{-55mm}\sup_{0\leq t\leq T^\star} ||w_{l} f^{n+1} (t)||_{L^\infty_p (L^2_x \cap L^{\infty}_x)} \hspace{0mm}\\ \leq ||w_{l} f_0||_{L^\infty_p (L^2_x \cap L^\infty_x)}+CT^\star \sup_{0\leq t\leq T} ||w_{l} f^{n} (t)||_{L^\infty_p (L^2_x \cap L^\infty_x)} \hspace{0mm}\\+CT^\star N^{{a \over 2}}\sup_{0\leq t\leq T^\star} ||w_{l} f^{n} (t)||_{L^\infty_p (L^2_x \cap L^\infty_x)}^2 \\ \hspace{22mm}+C\left[N^{-1}+N^{|\gamma|-2+O(\delta)}\right]\sup_{0\leq t \leq T^\star}||w_{l} f^{n} (t)||_{L^\infty_p (L^2_x \cap L^{\infty}_x)}^2 \\ \hspace{6mm} \leq {\widetilde{M} \over 2}+C\left[\widetilde{M}+\widetilde{M}^2\right]N^{a \over 2} T^\star+C\widetilde{M}^2 \left[N^{-1}+N^{|\gamma|-2+O(\delta)}\right].
\end{gathered}
\end{eqnarray*}
Above, by first choosing $N$ large, and second choosing $T^\star$ sufficiently small, we obtain $$\sup_{0\leq t\leq T^\star} ||w_{l} f^{n+1} (t)||_{L^\infty_p(L^2_x \cap L^\infty_x)} \leq \widetilde{M}.$$

For the remaining assertions we refer to the proof of Theorem 3.1 in [6].
\end{proof}
\noindent {\it{Remark.}} From (\ref{la1}) we may remove the assumption (\ref{ls}) for the case $a\in [0,2] \cap [0,2+\gamma).$
\begin{center}{
\section{Global existence}
}\end{center}
For the local solution $f(t,x,p)$ constructed in Theorem 1.5 or Theorem 3.1, we define 
\begin{equation}
h(t,x,p)=h_l (t,x,p)=w_l (p) f(t,x,p).\nonumber
\end{equation}
The mild form for $h(t,x,p)$ is given by
\begin{equation}
\label{inth}
\begin{gathered}
\hspace{-60mm}h(t,x,p)=e^{-\nu(p)t}h_0 \left(x-\hat{p}t,p\right)\\+\int_0^t e^{-\nu(p)(t-s)} K_l (h) \left(s,x-\hat{p}(t-s),p\right)\,ds\\+\int_0^t e^{-\nu(p)(t-s)}\Gamma_l  (h,h) \left(s,x-\hat{p}(t-s),p\right)\,ds.
\end{gathered}
\end{equation}
Here $h_0 (x,p)=w_l (p) f_0,$ and
\begin{equation*}
K_l (h)=w_l K\left(h \over w_l \right),\quad \Gamma_l (h,h)=w_l \Gamma \left({h \over w_l},{h \over w_l}\right).
\end{equation*}
Similarly for $K^{\chi}_l,\,K^{1-\chi}_l.$ We also define $$k_l(p,q)=k(p,q) {w_l (p) \over w_l (q)},\quad k^{\chi}_l(p,q)=k^{\chi}(p,q) {w_l (p) \over w_l (q)}.$$
From (\ref{inth}), we have
\begin{equation}
\begin{gathered}
\left|\left|h \right|\right|_{L^\infty_x} (t,p)\leq e^{-\nu(p) t}\left|\left|h_0 \right|\right|_{L^\infty_x} (p)+ \int_0^t e^{-\nu(p)(t-s)} \left|K_l (\left|\left|h \right|\right|_{L^\infty_x})(s,p)\right|\,ds\\+ \int_0^t e^{-\nu(p)(t-s)} \left|\Gamma_l \left(\left|\left|h \right|\right|_{L^\infty_x},\left|\left|h \right|\right|_{L^\infty_x}\right)(s,p)\right|\,ds
\end{gathered}
\end{equation}
Above the second term on the r.h.s. is split into
\begin{equation}
\begin{gathered}
\hspace{-20mm}\int_0^t e^{-\nu(p)(t-s)} \int_{\mathbb{R}^3}\left| k^\chi_l \left(p,q \right) \right| \left|\left|h \right|\right|_{L^\infty_x}(s,q)\,dq\,ds\\ \hspace{20mm}+\int_0^t e^{-\nu(p)(t-s)} w_l (p) \left|K^{1-\chi}_l \left({\left|\left|h \right|\right|_{L^\infty_x}}\right)(s,p) \right|\,ds.
\end{gathered}
\end{equation}
Similarly, we also have using Minkowski's inequality and an interpolation
\begin{equation}
\begin{gathered}
\hspace{-50mm}\left|\left|h \right|\right|_{L^2_x} (t,p)\leq e^{-\nu(p) t}\left|\left|h_0 \right|\right|_{L^2_x} (p)\\ \hspace{11mm}+\int_0^t e^{-\nu(p)(t-s)} \int_{\mathbb{R}^3} \left|k^\chi_l \left(p,q \right)\right|\left|\left|h \right|\right|_{L^2_x}(s,q)\,dq\,ds\\ \hspace{-1.5mm}+\int_0^t e^{-\nu(p)(t-s)}\left|K^{1-\chi}_l \left(\left|\left|h \right|\right|_{L^2_x}\right)(s,p)\right|\,ds \\\hspace{-5.5mm}+ \int_0^t e^{-\nu(p)(t-s)} \left|\left|\Gamma_l \left(h,h \right)(s,p)\right|\right|_{L^1_x}\,ds\\ \hspace{-4mm}+ \int_0^t e^{-\nu(p)(t-s)} \left|\left|\Gamma_l \left(h,h \right)(s,p)\right|\right|_{L^\infty_x}\,ds.
\end{gathered}
\end{equation}

Notice that for hard potentials,
\begin{equation}
\int_0^t e^{-\nu(p)(t-s)}\nu^i (p) \,ds\leq 2, \quad i=0,1. \nonumber
\end{equation}
We will use the following integral.
\begin{prop} Let $l \geq 0$ and $\zeta=\min\left\{2\zeta_a ,2\zeta_b+ {{a+b} \over 2}\right\}>0.$
\begin{equation}
\int_{\mathbb{R}^3} \left|k^{\chi}_l \left(p,q\right)\right|\,d q\leq C_\chi \nu(p) (p^0)^{-\zeta},
\end{equation}
\begin{equation}
\int_{\mathbb{R}^3} \left| k(p,q)\right| e^{-cq^0}\,d q \leq C.
\end{equation}
\end{prop}
\begin{proof}
It is easy to verify that 
\begin{equation}
\left(p^0 q^0 \right)^\alpha e^{-c|p-q|}\lesssim \left(p^0 \right)^{2\alpha} e^{-{c \over 2} |p-q|},
\end{equation}
for $\alpha \in \mathbb{R}^3.$ Moreover, it is well known that
$${{\left\{|p \times q|^2 +|p-q|^2 \right\}^{1 \over 2}} \over {\sqrt{p^0 q^0}}}\leq g(p,q) \leq |p-q|.$$
See for instance [6].
We use Lemma 2.2 and Lemma 2.3. By calculus, we can get (4.5) and (4.6).
\end{proof}

We will prove the following estimate by using Lemmas \ref{lemg1} and \ref{lemg2}. For simplicity, we define $\mathcal{E}_0=(\left|M_0\right|+\left|E_0\right|+\left|H_0\right|)^k$ when $\mathcal{E}_0 \ll 1.$ Here $k={{1 \over 10}\left(1-{1 \over d}\right)}$ and $d$ is given in Lemma 2.6.
%lem4.2
\begin{lem} Let $l >14+{a \over 2}.$ There exists a constant $C^\star>0$ such that
\begin{equation}
\begin{gathered}
\hspace{0mm}\sup_{0 \leq s \leq t} \left|\left|h(s) \right|\right|_{L^\infty_{p} \left(L^2_x \cap L^\infty_x \right)} \leq C^\star M^2+C^\star\mathcal{E}_0 \left\{1+\sup_{0 \leq s \leq t} \left|\left|h(s) \right|\right|^{1 \over 2}_{L^\infty_{x,p}}\right\}\\+C^\star \sup_{T^\star \leq s \leq t} \left\{ \left|\left|h(s) \right|\right|^{{9d+1}\over 5d}_{L^\infty_{x,p}} \sup_{y \in \mathbb{R}^3} \left(\int_{\mathbb{R}^3} \left|f(s,y,q)\right|\,dq\right)^{{d-1} \over 5d}\right\}\\+C^\star\sup_{T^\star \leq s \leq t} \left\{ \left|\left|h(s) \right|\right|^{{9d+1}\over 5d}_{L^\infty_{p} L^2_x} \sup_{q \in \mathbb{R}^3} \left(\int_{\mathbb{R}^3} \left|\left|f\right|\right|_{L^2_x} (s,q)\,dq\right)^{{d-1} \over 5d}\right\}.\nonumber
\end{gathered}
\end{equation}
\begin{proof}
From Lemma 4.2 in [6],
\begin{equation}
\begin{gathered}
\hspace{-25mm}\sup_{0 \leq s \leq t} \left|\left|h(s) \right|\right|_{L^\infty_{x,p}} \lesssim M^2+\mathcal{E}_0\\ \hspace{25mm}+\sup_{T^\star \leq s \leq t} \left\{ \left|\left|h(s) \right|\right|^{{9d+1}\over 5d}_{L^\infty_{x,p}} \sup_{y \in \mathbb{R}^3} \left(\int_{\mathbb{R}^3} \left|f(s,y,q)\right|\,dq\right)^{{d-1} \over 5d}\right\}.\nonumber
\end{gathered}
\end{equation}
We need only to show that
\begin{equation}
\begin{gathered}
\hspace{-15mm}\sup_{0 \leq s \leq t} \left|\left|h(s) \right|\right|_{L^\infty_{p} L^2_x} \lesssim M^2+\mathcal{E}_0 \left\{1+\sup_{0 \leq s \leq t} \left|\left|h(s) \right|\right|^{1 \over 2}_{L^\infty_{x,p}}\right\} \\ \hspace{15mm}+\sup_{T^\star \leq s \leq t} \left\{ \left|\left|h(s) \right|\right|^{{9d+1}\over 5d}_{L^\infty_{p} L^2_x} \sup_{q \in \mathbb{R}^3} \left(\int_{\mathbb{R}^3} \left|\left|f\right|\right|_{L^2_x} (s,q)\,dq\right)^{{d-1} \over 5d}\right\},
\end{gathered}
\end{equation}
Now we estimate the r.h.s. of (4.4). First, again using (4.4), the second term is split into
\begin{equation}
\begin{gathered}
\int_0^t e^{-\nu(p)(t-s)}\int_{\mathbb{R}^3} \left|k^{\chi}_l \left(p,q\right)\right|e^{-\nu(q)s} \left|\left|h_0\right|\right|_{L^2_x} (q)\,dq\,ds
\\+\int_0^t e^{-\nu(p)(t-s)} \int_{\mathbb{R}^3} \left|k^\chi_l \left(p,q\right)\right| \\ \times \left\{\int_0^s e^{-\nu(q)(s-s_1)} \int_{\mathbb{R}^3} \left|k^\chi_l \left(q,q_1\right)\right|\, \left|\left|h\right|\right|_{L^2_x} (s_1 , q_1)\,d q_1\,d s_1\right\}\,d q\,d s\\+\int_0^t e^{-\nu(p)(t-s)} \int_{\mathbb{R}^3} \left|k^\chi_l (p,q)\right|\\ \times \left\{ \int_0^s e^{-\nu(q)(s-s_1)} \left| K^{1-\chi}_l \left(\left|\left|h\right|\right|_{L^2_x}\right)(s_1,q) \right|\,d s_1\right\}\,d q\,d s\\+\int_0^t e^{-\nu(p)(t-s)} \int_{\mathbb{R}^3} \left|k^\chi_l \left(p,q\right)\right| \\ \times \left\{\int_0^s e^{-\nu(q)(s-s_1)} \left|\left|\Gamma_l \left(h,h \right) \right|\right|_{L^1_x} (s_1,q)\,d s_1\right\}\,d q\,d s\\+\int_0^t e^{-\nu(p)(t-s)} \int_{\mathbb{R}^3} \left|k^\chi_l \left(p,q\right)\right| \\ \times \left\{\int_0^s e^{-\nu(q)(s-s_1)} \left|\left|\Gamma_l \left(h ,h\right) \right|\right|_{L^\infty_x} (s_1,q)\,d s_1\right\}\,d q\,d s.
\end{gathered}
\end{equation}

Clearly, the first term of (4.9) is bounded by
\begin{equation}
\begin{gathered}
C_\chi \left|\left|h_0\right|\right|_{L^\infty_p L^2_x}. \nonumber
\end{gathered}
\end{equation}
By Lemma 2.5, for any $\eta>0$ the third term of (4.9) is bounded by
\begin{equation}
\begin{gathered}
\int_0^t e^{-\nu(p)(t-s)}  \int_0^s e^{\nu(q)(s-s_1)} \int_{\mathbb{R}^3} \left|k_l \left(p,q\right)\right| e^{-cq^0}\left|\left|h\right|\right|_{L^\infty_{q_1} L^2_x}(s_1)\,d q\,ds_1\,ds\\ \leq C \eta \sup_{0 \leq s \leq t} \left|\left|h(s) \right|\right|_{L^\infty_{q_1} L^2_x},\nonumber
\end{gathered}
\end{equation}
By using Lemma 2.6, the last two terms of (4.9) is bounded by
\begin{equation}
\begin{gathered}
C_\chi \sup_{0 \leq s \leq t} \left\{ \left|\left|h(s) \right|\right|^{{9d+1}\over 5d}_{L^\infty_{p} L^2_x} \left(\int_{\mathbb{R}^3} \left|\left|f\right|\right|_{L^2_x} (s,q)\,dq\right)^{{d-1} \over 5d}\right\}\\+C_\chi \sup_{0 \leq s \leq t} \left\{||w_l f(s)||^{{9d+1} \over 5d}_{L^{\infty}_{x,p}} \sup_{y \in \mathbb{R}^3} \left(\int_{\mathbb{R}^3} |f(s,y,q)|\,dq\right)^{{d-1} \over 5d}\right\}.\nonumber
\end{gathered}
\end{equation}
If $l>3$ then Theorem 3.1 yields
\begin{equation}
\begin{gathered}
\sup_{0 \leq s \leq T^\star} \left\{ \left|\left|h(s) \right|\right|^{{9d+1}\over 5d}_{L^\infty_{p} L^2_x} \left(\int_{\mathbb{R}^3} \left|\left|f\right|\right|_{L^2_x} (s,q)\,dq\right)^{{d-1} \over 5d}\right\}\\
\leq \sup_{0 \leq s \leq T^\star} \left\{ \left|\left|h(s) \right|\right|^{2}_{L^\infty_{p} L^2_x} \left(\int_{\mathbb{R}^3}  w_{-l}(q)\,dq\right)^{{d-1} \over 5d}\right\} \\ \leq C \left|\left|h_0 \right|\right|^2_{L^\infty_p L^2_x}
\nonumber
\end{gathered}
\end{equation}
Similarly, we have 
\begin{equation}
\begin{gathered}
\sup_{0 \leq s \leq T^\star} \left\{ \left|\left|h(s) \right|\right|^{{9d+1}\over 5d}_{L^\infty_{x,p}} \sup_{y \in \mathbb{R}^3} \left(\int_{\mathbb{R}^3} \left|f(s,y,q)\right| \,dq\right)^{{d-1} \over 5d}\right\} \\ \leq C \left|\left|h_0 \right|\right|^2_{L^\infty_{x,p}}.\nonumber
\end{gathered}
\end{equation}
Also for the other terms of (4.4), we can get estimates similar to the estimates for each term of (4.9).

Next, we estimate the second term of (4.9), which is split into the following three cases.

Case 1. For $|p| \geq N.$ By (4.5) the second term is bounded by
\begin{equation}
\begin{gathered}
{C_\chi {N^{-\zeta}}} \int_0^t e^{-\nu(p)(t-s)} \nu(p) \int_0^s e^{-\nu(q)(s-s_1)}\nu(q)\left(q^0\right)^{-\zeta}\left|\left|h(s_1)\right|\right|_{L^\infty_p L^2_x}\,ds_1\,ds\\ \leq C_\chi N^{-\zeta} \sup_{0 \leq s \leq t} \left|\left|h(s)\right|\right|_{L^\infty_p L^2_x}. \nonumber
\end{gathered}
\end{equation}

Case 2. For $|p| \leq N,\,|q|\geq 2N,$ or $|q| \leq 2N,\,|q_1| \geq 3N.$ Then $|p-q| \geq N$, and then use Lemma 2.3 to get 
\begin{equation}
\begin{gathered}
\left|k^\chi_l \left(p,q \right)\right| \leq e^{-{c \over 2}N}\left|k^\chi_l (p,q)\right| e^{{c \over 2}|p-q|} \leq C_\chi e^{-{c \over 2} N} e^{-{c \over 2}{{|p-q|}}},
\nonumber
\end{gathered}
\end{equation}
and similarly,
\begin{equation}
\begin{gathered}
\left|k^\chi_l \left(q,q_1 \right)\right| \leq C_\chi e^{-{c \over 2} N} e^{-{c \over 2}{{|q-q_1|}}}.\nonumber
\end{gathered}
\end{equation}
We have
\begin{equation}
\begin{gathered}
\int_0^t e^{-\nu(p)(t-s)} \int_{\mathbb{R}^3} \left|k^\chi_l \left(p,q\right)\right| \int_0^s e^{-\nu(q)(s-s_1)} \int_{\mathbb{R}^3} \left|k^\chi_l \left(q,q_1\right)\right|\, \\ \times \left[\chi_{|p| \leq N} \chi_{|q| \geq 2N}+\chi_{|q| \leq 2N} \chi_{|q_1| \geq 3N}\right]\left|\left|h\right|\right|_{L^2_x} (s_1 , q_1)\,d q_1\,d s_1\,d q\,d s\\ \leq C_{\chi} e^{-{c \over 2}N} \sup_{0 \leq s \leq t}\left|\left|h (s)\right|\right|_{L^\infty_p L^2_x}.\nonumber
\end{gathered}
\end{equation}
where $\chi_A$ is the characteristic function of $A$.

Case 3. For $|p| \leq N,\,|q| \leq 2N,\, |q_1| \leq 3N.$ %For any small $\kappa>0,$ we further split 
The remaining part is
\begin{equation}
\begin{gathered}
\int_0^t e^{-\nu(p)(t-s)} \int_{|q| \leq 2N} \left|k^\chi_l \left(p,q\right)\right| \int_0^{s} e^{-\nu(q)(s-s_1)} \int_{|q_1| \leq 3N} \left|k^\chi_l \left(q,q_1\right)\right|\, \\ \times\,\, \chi_{|p| \leq N} \left|\left|h\right|\right|_{L^2_x} (s_1 , q_1)\,d q_1\,d s_1\,d q\,d s. 
%\\+\int_0^t e^{-\nu(p)(t-s)} \int_{|q| \leq 2N} \left|k^\chi_l \left(p,q\right)\right| \int_{s-\kappa}^s e^{-\nu(q)(s-s_1)} \int_{|q_1| \leq 3N} \left|k^\chi_l \left(q,q_1\right)\right|\, \\ \times\,\, \chi_{|p| \leq N} \left|\left|h\right|\right|_{L^2_x} (s_1 , q_1)\,d q_1\,d s_1\,d q\,d s.
\end{gathered}
\end{equation}
Since $k^\chi_l \left(p,q\right)$ is bounded, the above is bounded by
\begin{equation}
\begin{gathered}
C_{\chi,N} \int_0^t  e^{-c(t-s)} \int_0^{s} e^{-c(s-s_1)}\int_{|q_1| \leq 3N} \left|\left|h\right|\right|_{L^2_x} (s_1 , q_1)\,d q_1\,d s_1 \, d s,\nonumber
\end{gathered}
\end{equation}
Moreover, we have
\begin{equation}
\begin{gathered}
\left[\int_{|q_1| \leq 3N} \left|\left|h\right|\right|_{L^2_x} (s_1 , q_1)\,d q_1\right]^2 \leq C_N \int_{|q_1| \leq 3N} \int_{\mathbb{R}^3} \left|h(s_1,y,q_1)\right|^2\,d y \,d q_1\\ \leq C_N \int_{|q_1| \leq 3N} \int_{\mathbb{R}^3} \left| f(s_1,y,q_1)\right|^2 \chi_{\left\{\left|f(s_1,y,q_1)\right| \leq \sqrt{J(q_1)}\right\}}\,d y\,d q_1\\+C_N \left|\left|h(s_1)\right|\right|_{L^\infty_{x,p}}\int_{|q_1| \leq 3N} \int_{\mathbb{R}^3}\sqrt{J(q_1)}\left| f(s_1,y,q_1)\right|\chi_{\left\{\left|f(s_1,y,q_1)\right| \geq \sqrt{J(q_1)}\right\}}\,dy\,dq_1.
\end{gathered}
\end{equation}
Applying Lemma 2.5 to the above, (4.10) is bounded by
\begin{equation}
\begin{gathered}
C_{\chi,N} \,\mathcal{E}_0\left\{1+\sqrt{ \sup_{0 \leq s \leq t} \left|\left|h(s)\right|\right|_{L^\infty_{x,p}}}\right\}.\nonumber
\end{gathered}
\end{equation}
%Lastly, the second term of (4.14) is bounded by 
%$$C_{\chi,N} \kappa \sup_{0 \leq s \leq t} \left|\left|h(s)\right|\right|_{L^\infty_{p} L^2_x}.$$
We obtain the desired estimate (4.8), for the above terms by first choosing $\eta$ small, and second choosing $N$ large.
\end{proof}
\end{lem}
Notice that, as in (4.11), for any $L>0$ we have
\begin{equation}
\begin{gathered}
\left|\left|h(s) \right|\right|^{{9d+1}\over 5d}_{L^\infty_{p} L^2_x}\left(\int_{\mathbb{R}^3} \left|\left|f\right|\right|_{L^2_x} (s,q)\,dq\right)^{{d-1} \over 5d}\\ \leq  \left|\left|h(s) \right|\right|^{{9d+1}\over 5d}_{L^\infty_{p} L^2_x} \left\{ \int_{|q| \geq L} +\int_{|q| \leq L} \left|\left|f \right|\right|_{L^2_x} (s,q) \,d q\right\}^{{d-1} \over 5d}\\ \leq C L^{{(-l+3)} \cdot {{d-1} \over 5d}} \sup_{0 \leq s \leq t} \left|\left|h(s) \right|\right|^{2}_{L^\infty_{p} L^2_x}+C_L \mathcal{E}_0\left\{1+\sup_{0 \leq s \leq t} \left|\left|h(s) \right|\right|^{{d-1} \over 10d}_{L^\infty_{x,p}}\right\}.
\end{gathered}
\end{equation}
%By Lemma 4.2 and the above, here exists $C^\star =C^\star>0$ such that
%\begin{equation}
%\begin{gathered}
%\hspace{-15mm}\sup_{0 \leq s \leq t} \left|\left|h(s) \right|\right|_{L^\infty_{p} \left(L^2_x \cap L^\infty_x \right)} \leq C^\star \left|\left|h_0 \right|\right|^2_{L^\infty_{p} \left(L^2_x \cap L^\infty_x \right)}+ \mathcal{E}_0 +{1 \over 10} \sup_{0 \leq s \leq t} \left|\left|h(s) \right|\right|^{{14d+1}\over 5d}_{L^\infty_{p} L^2_x} \\+C^\star \sup_{T^\star \leq s \leq t} \left\{ \left|\left|h(s) \right|\right|^{{9d+1}\over 5d}_{L^\infty_{x,p}} \sup_{y \in \mathbb{R}^3} \left(\int_{\mathbb{R}^3} \left|f(s,y,q)\right|\,dq\right)^{{d-1} \over 5d}\right\}.
%\end{gathered}
%\end{equation}

We also need the following lemma.
%lem4.3
\begin{lem}\cite{r6}. Let $l>5.$ For any $\eta>0,$ we have
\begin{equation}
\begin{gathered}
\int_{\mathbb{R}^3} \left| f\left(t,x,p\right)\right|\,dp \leq \int_{\mathbb{R}^3} e^{-\nu(p)t}\left| f_0\left(x-\hat{p}t,p\right)\right|\,dp\\+C\eta\sup_{0 \leq s \leq t} \left\{\left|\left|h(s)\right|\right|_{L^\infty_{x,p}}+\left|\left|h(s)\right|\right|^2_{L^\infty_{x,p}}\right\}\\+C_{\eta}\mathcal{E}_0 \left\{1+\sup_{0 \leq s \leq t}\left|\left|h(s)\right|\right|^{1+{1 \over d}}_{L^\infty_{x,p}}+\sup_{0 \leq s \leq t}\left|\left|h(s)\right|\right|^{1+{{9d+1} \over 5d}}_{L^\infty_{x,p}}\right\}.\nonumber
\end{gathered}
\end{equation}
\end{lem}

We are now ready to prove Theorem 1.1.
\begin{proof}[Proof of Theorem 1.1]
Assume that $t>T^\star$ and
\begin{equation}
\begin{gathered}
\sup_{0 \leq s \leq t} \left|\left|h(s)\right|\right|_{L^\infty_p \left(L^2_x \cap L^\infty_x \right)} \leq 2C^\star M^2 \equiv 2R_0,\nonumber
\end{gathered}
\end{equation}
where $C^\star>1$ is given in Lemma 4.2.
Then from Lemma 4.2 and (4.12) we have
\begin{equation}
\label{gg1}
\begin{gathered}
\sup_{0 \leq s \leq t} \left|\left|h(s)\right|\right|_{L^\infty_p \left(L^2_x \cap L^\infty_x \right)} \leq R_0+3R_0 \mathcal{E}_0 +4C {R_0}^2 L^{-{{d-1} \over 5d}} +3C_L R_0 \mathcal{E}_0  \\+(2R_0)^{{9d+1} \over {5d}} \sup_{{T^\star \leq s \leq t},\,y \in \mathbb{R}^3} \left(\int_{\mathbb{R}^3} \left|f(s,y,q)\right|\,d q\right)^{{d-1} \over 5d}.\nonumber
\end{gathered}
\end{equation}
By Lemma 4.3 we also have
\begin{equation}
\begin{gathered}
\int_{\mathbb{R}^3} \left| f\left(s,x,p\right)\right|\,dp \leq \int_{|p| \leq L} e^{-\nu(p)s}\left| f_0\left(x-\hat{p}s,p\right)\right|\,dp\\+L^{-l+3}M+6C\eta {R_0}^2+9C_{\eta} {R_0}^2 \mathcal{E}_0. 
\end{gathered}
\end{equation}
We first choosing $\eta$, $L^{-1}$ and 
\begin{equation}
\sup_{T^\star  \leq s \leq t,\,x \in \mathbb{R}^3} \int_{|p| \leq L} e^{-\nu(p)s}\left| f_0\left(x-\hat{p}s,p\right)\right|\,dp\nonumber
\end{equation}
small, and second choosing $\mathcal{E}_0$ sufficiently small, so that $$\sup_{0 \leq s \leq t} \left|\left|h(s)\right|\right|_{L^\infty_p \left(L^2_x \cap L^\infty_x \right)} \leq R_0+4 \times {R_0 \over 5} \leq {9 \over 5}R_0.$$
This proves the theorem.
\end{proof}

\begin{center}{
\section{Nonlinear decay theory}
}\end{center}
To prove Theorems 1.2, we will need linear decay estimates.
The semigroup $U_l (t) g_0$ denotes the solution to the following linearized equation
\begin{eqnarray}
  \left\{
    \begin{array}{l}
       \left[\partial_t +\hat{p} \cdot \nabla_x +\nu(p)-K_l\right]g=0, \\
       g(0,x,p)=g_0 (x,p).
    \end{array}
  \right.
\end{eqnarray}
We also denote $U(t)=U_0 (t)$. 
It is easy to check that
\begin{equation}
\begin{gathered}
U_l (t)\left\{ w_l g_0\right\}=w_l (p) U(t) g_0,\\U_l (t) \left\{w_l g_0\right\}=w_{a \over 2} (p) U_{l-{a \over 2}} (t) \left[w_{l-{a \over 2}} g_0\right].
\end{gathered}
\end{equation}
Recall the notation
\begin{equation*}
\varpi (t)=(1+t)^{-\sigma_r},\quad \sigma_r=- {3 \over 2} \left({1 \over r}-{1 \over 2}\right), \quad 1 \leq r \leq 2.
\end{equation*}

\begin{thm} Fix $l \geq 0,\,r \in [1,2].$ Suppose that initially we have $w_l g_0 \in L^\infty_p \left(L^2_x \cap L^\infty_x \right),$ then the semigroup satisfies
\begin{equation}
\begin{gathered}
\left|\left|U_l (t) g_0\right|\right|_{L^\infty_{p} \left(L^2_x \cap L^\infty_x\right)}\lesssim (1+t)^{-\sigma_r} \left(\left|\left|w_{l}g_0\right|\right|_{L^\infty_{p} \left(L^2_x \cap L^\infty_x\right)}+\left|\left|g_0\right|\right|_{L^2_{p} L^r_x}\right).
\end{gathered}
\end{equation}
\end{thm}

Moreover, by an interpolation, we have
\begin{equation}
\begin{gathered}
\left|\left|U_l (t) g_0\right|\right|_{L^\infty_{p} \left(L^2_x \cap L^\infty_x\right)}\lesssim (1+t)^{-\sigma_r} \left|\left|w_{l}g_0\right|\right|_{L^\infty_{p} \left(L^1_x \cap L^\infty_x\right)}.
\end{gathered}
\end{equation}

\begin{proof}[Proof of Theorem 5.1] This case is much easier to handle than soft potentials. First, we can write the solution $g(t)=U_l (t) g_0$ to (5.1), as
\begin{equation}
\begin{gathered}
g(t,x,p)=e^{-\nu(p)t} g_0 \left(x-\hat{p}t,p\right) +\int_0^t d s_1\,e^{-\nu(p)(t-s_1)} K^{1-\chi} (g) (s_1,y_1,p)\\
%3
+\int_0^t d s_1\, e^{-\nu(p)(t-s_1)}\int_{\mathbb{R}^3} d q_1\, k^\chi (p,q_1)e^{-\nu(p)s_1} f_0 (y_1 -\hat{q_1} s_1,q_1)\\
%4
\hspace{-5mm}+\int_{\mathbb{R}^3} dq_1 \, k^\chi (p,q_1) \int_0^t d s_1 \int_0^{s_1} d s_2 \, e^{-\nu(p)(t-s_1)} e^{-\nu(q_1)(s_1-s_2)} \\ \hspace{70mm} \times \,K^{1-\chi}(g)(s_2,y_2,q_1)\\
%5
\hspace{-9mm}+\int_{\mathbb{R}^3} d q_1\, k^\chi (p,q_1) \int_{\mathbb{R}^3} d q_2\, k^\chi (q_1,q_2) \int_0^t d s_1 \,e^{-\nu(p)(t-s_1)} \\ \hspace{40mm} \times \int_0^{s_1} d s_2\,e^{-\nu(q_1)(s_1 -s_2)} (1-\chi_{R_N})\,g (s_2 ,y_2 , q_2)\\
%6
\hspace{-9mm}+\int_{\mathbb{R}^3} d q_1\, k^\chi (p,q_1) \int_{\mathbb{R}^3} d q_2\, k^\chi (q_1,q_2) \int_0^t d s_1 \,e^{-\nu(p)(t-s_1)} \\ \hspace{49mm} \times \int_0^{s_1} d s_2\,e^{-\nu(q_1)(s_1 -s_2)} \chi_{R_N} \,g (s_2 ,y_2 , q_2),
\end{gathered}
\end{equation}
where $\chi_{R_N}$ is the characteristic function of $\left\{|p| \leq N\,|q_1| \leq 2N ,\, |q_2| \leq 3 N\right\}$, and
\begin{equation}
\begin{gathered}
y_1=x-\hat{p}(t-s_1),\\y_2=y_1-\hat{q_1}(s_1 -s_2)=x-\hat{p}(t-s_1)-\hat{q_1} (s_1 -s_2).\nonumber
\end{gathered}
\end{equation}

Now we estimate the $L^\infty_p ( L^2_x \cap L^\infty_x)$ norm of (5.5) multiplied by $w_l(p)$. For the first five terms of (5.5), we can estimate them as in each term of (4.9). The norms of these terms multiplied by $w_l(p)$ are bounded by
\begin{equation}
C_\chi (1+t)^{-\sigma_r}\left|\left| w_l g_0 \right|\right|_{L^\infty_{p} \left(L^2_x \cap L^\infty_x \right )},\nonumber
\end{equation}
or when $N$ large enough, for any $\eta >0$,
\begin{equation}
\eta (1+t)^{-\sigma_r}\left|\left|\varpi_r w_l g \right|\right|_{L^\infty_{p,t} \left(L^2_x \cap L^\infty_x \right )}.\nonumber
\end{equation}
Notice, in particular, that $w_l (p) e^{-c|p-q_1|} \lesssim w_l (q_1 ) e^{-{c \over 2}{|p-q_1|}}$ from (4.7).

We only estimate the last term of (5.5). For $\kappa>0$ we clearly have
\begin{equation}
\begin{gathered}
\hspace{-15mm}\left|\left|w_l (p)\int_{\mathbb{R}^3} d q_1\, k^\chi (p,q_1) \int_{\mathbb{R}^3} d q_2\, k^\chi (q_1,q_2) \int_\kappa^t d s_1 \,e^{-\nu(p)(t-s_1)} \right.\right. \\ \hspace{35mm} \left.\left.\times \int_{s_1 - \kappa}^{s_1} d s_2\,e^{-\nu(q_1)(s_1 -s_2)} \chi_{R_N} \,g (s_2 ,y_2 , q_2) \right|\right|_{L^2_x \cap L^\infty_x}\\ \hspace{-10mm}+\left|\left|w_l (p)\int_{\mathbb{R}^3} d q_1\, k^\chi (p,q_1) \int_{\mathbb{R}^3} d q_2\, k^\chi (q_1,q_2) \int_0^\kappa d s_1 \,e^{-\nu(p)(t-s_1)} \right.\right. \\ \hspace{35mm} \left.\left.\times \int_{0}^{s_1} d s_2\,e^{-\nu(q_1)(s_1 -s_2)} \chi_{R_N} \,g (s_2 ,y_2 , q_2) \right|\right|_{L^2_x \cap L^\infty_x}
\\ \leq C_{\chi,N} \kappa (1+t)^{-\sigma_r} \left|\left| \varpi_r w_l g \right| \right|_{L^\infty_{q,t} \left(L^2_x \cap L^\infty_x\right)}.\nonumber
\end{gathered}
\end{equation}
For the remaining part of the last term of (5.5) we use the following $L^2$ decay estimate :
\begin{equation}
\left|\left| U (t) g_0 \right|\right|_{L^2_{x,p}} \lesssim (1+t)^{-\sigma_r} \left|\left|g_0 \right|\right|_{L^2_p \left(L^2_{x} \cap L^r_x\right)}, \quad l \geq 0,\,r \in [1,2]. \nonumber
\end{equation}
This assertion follows as in [7]. We also refer to Section 2 of [5] (which handles soft potentials).
Then, we have
\begin{equation}
\begin{gathered}
\hspace{-15mm}\left|\left|w_l (p)\int_{|q_1| \leq 2N} d q_1\, k^\chi (p,q_1) \int_{|q_2| \leq 3N} d q_2\, k^\chi (q_1,q_2) \int_{\kappa}^t d s_1 \,e^{-\nu(p)(t-s_1)} \right.\right. \\ \hspace{35mm} \left.\left.\times \int_0^{s_1 - \kappa} d s_2\,e^{-\nu(q_1)(s_1 -s_2)} \chi_{R_N} \,g (s_2 ,y_2 , q_2) \right|\right|_{L^2_x \cap L^\infty_x}\\ \leq C_{\kappa,N} \int_\kappa^t ds_1\,e^{-c(t-s_1)} \int_0^{s_1 -\kappa} ds_2 \, e^{-c(s_1-s_2)}\left|\left| g (s_2) \right|\right|_{L^2_{x,p}}\\ \leq C_{\kappa,N} \left|\left|g_0 \right|\right|_{L^2_p \left(L^2_{x} \cap L^r_x\right)} \int_0^t e^{-{c \over 2}(t-s_2)} (1+s_2)^{-\sigma_r}\,d s_2\\ \leq C_{\kappa,N} (1+t)^{-\sigma_r}\left(\left|\left|w_l g_0 \right|\right|_{L^\infty_p L^2_x}+\left|\left|g_0 \right|\right|_{L^2_p L^r_x}\right).\nonumber
\end{gathered}
\end{equation}
Here we have used, for the $L^2_x$ norm, Cauchy-Schwartz inequality  for the $q_2$ integral :
\begin{equation}
\begin{gathered}
\left|\int_{|q_2| \leq 3N} d q_2 \,k^\chi (q_1, q_2) g(s_2,y_2,q_2) \right| \leq C_{\chi,N} \left|\left|g\right|\right|_{L^2_{q_2}} (s_2, y_2).\nonumber
\end{gathered}
\end{equation}
For the $L^\infty_x$ norm, we have used the change of variable $q_1 \rightarrow y_2$. Note that $$d q_1={(1+|q_2 |^2)^{5 \over 2} \over (s_1 -s_2)} dy_2.$$
By first choosing $\eta$ and $N^{-1}$ small, and second choosing $\kappa$ sufficiently small, we obtain (5.3).
\end{proof}

\begin{proof}[Proof of Theorem 1.3] We can write $h=w_l f$ where $f$ is constructed in Theorem 1.1, as 
\begin{equation}
\begin{gathered}
h(t,x,p)=\left\{U_l (t) h_0 \right\}(x,p)+\int_0^t \left\{U_l (t-s) w_l \Gamma (f(s),f(s)) (x,p)\right\}\,ds.
\end{gathered}
\end{equation}
Moreover, we can split the second term into
\begin{equation}
\begin{gathered}
\int_0^t e^{-\nu(p)(t-s)} \left\{w_l (p) \Gamma (f(s),f(s)) \right\}(x,p)\, ds\\+\int_0^t \int_s^t e^{-\nu(p)(t-s_1)} \left\{K^\chi_l U (s_1 -s) w_l \Gamma (f(s),f(s)) \right\}(x,p) \,ds_1 \,ds\\+\int_0^t \int_s^t e^{-\nu(p)(t-s_1)} \left\{K^{1-\chi}_l U (s_1 -s) w_l \Gamma (f(s),f(s)) \right\}(x,p) \,ds_1 \,ds.
\end{gathered}
\end{equation}
We denote each term of (5.7) by $E_1,\,E_2$ and $E_3$ in order from the top. 

First, from the conditions on initial data, (4.12) and (4.13), for any $\eta>0$ we can obtain
\begin{equation}
\begin{gathered}
\sup_{t,\,x} \int_{\mathbb{R}^3} \left|f(t,x,p)\right|\, dp+\sup_t \int_{\mathbb{R}^3} \left|\left| f \right|\right|_{L^2_x} (t,p)\,dp \leq {\eta}^{5d \over {d-1}}.
\end{gathered}
\end{equation}
Since $h(t,x,p)$ is bounded, by using Lemma 2.6 we have
\begin{equation}
\begin{gathered}
\left|\left| E_1 \right|\right|_{L^2_x \cap L^\infty_x} \leq C \int_0^t e^{-\nu(p)(t-s)} w_l (p) \left|\left| \Gamma \left(f(s),f(s)\right) \right|\right|_{L^1_x \cap L^\infty_x}\,ds\\ \leq C \eta \int_0^t e^{-\nu(p)(t-s)} \nu(p) \left|\left| h(s) \right|\right|^{{9d+1} \over 5d}_{L^\infty_p \left(L^1_x \cap L^\infty_x \right)}\,ds\\ \leq C_M \eta \left|\left| \varpi_r h \right|\right|_{L^\infty_{t,p} \left(L^2_x \cap L^\infty_x \right)} \int_0^t e^{-\nu(p)(t-s)} \nu(p) (1+s)^{-\sigma_r}\,ds\\ \leq C_M \eta (1+t)^{-\sigma_r} \left|\left| \varpi_r h \right|\right|_{L^\infty_{t,p} \left(L^2_x \cap L^\infty_x \right)}.
\end{gathered}
\end{equation}
For $E_2$ and $E_3$ we use (5.2) and (5.4). These and (4.7) yield
\begin{equation}
\begin{gathered}
\left|\left|K^\chi_l U (s_1 -s) w_l \Gamma \left(f(s),f(s)\right)\right|\right|_{L^2_x \cap L^\infty_x}\\ \leq C\int_{\mathbb{R}^3} \left| k^\chi_l (p,q) \right| \left(q^0 \right)^{a \over 2}dq
\left|\left| U_{l-{a \over 2}} (s_1 -s) w_{l-{a \over 2}} \Gamma \left(f(s),f(s)\right) \right|\right|_{L^\infty_{q_1} \left(L^2_x \cap L^\infty_x \right)}\\ \leq C\int_{\mathbb{R}^3} \left| k^\chi_l (p,q) \right| \left(p^0 \right)^{a \over 2}dq
\left|\left| U_{l-{a \over 2}} (s_1 -s) w_{l-{a \over 2}} \Gamma \left(f(s),f(s)\right) \right|\right|_{L^\infty_{q_1} \left(L^2_x \cap L^\infty_x \right)}\\ \leq C_{\chi} (1+s_1 -s)^{-\sigma_r}\nu(p)\left|\left| \nu^{-1} w_{l} \Gamma \left(f(s),f(s)\right) \right|\right|_{L^\infty_{q_1} \left(L^1_x \cap L^\infty_x \right)}.\nonumber
\end{gathered}
\end{equation}
From combining (5.2) and (5.4) with Lemma 2.4 we get
\begin{equation}
\begin{gathered}
\left|\left|K^{1-\chi}_l U (s_1 -s) w_l \Gamma \left(f(s),f(s)\right)\right|\right|_{L^2_x \cap L^\infty_x}\\ \leq K^{1-\chi}_l \left(\left|\left| U (s_1 -s) w_l \Gamma \left(f(s),f(s)\right)\right|\right|_{L^2_x \cap L^\infty_x} \right)\\ \leq K^{1-\chi}_l  \left(\left|\left| w_{a \over 2} U_{l-{a \over 2}} (s_1 -s) w_{l-{a \over 2}} \Gamma \left(f(s),f(s)\right)\right|\right|_{L^2_x \cap L^\infty_x} \right)\\ \leq \eta \left|\left|U_{l-{a \over 2}} (s_1 -s) w_{l-{a \over 2}} \Gamma \left(f(s),f(s)\right)\right|\right|_{L^\infty_{q_1} \left(L^2_x \cap L^\infty_x\right)}\\ \leq C \eta (1+s_1 -s)^{-\sigma_r}\left|\left|\nu^{-1}w_{l} \Gamma \left(f(s),f(s)\right)\right|\right|_{L^\infty_{q_1} \left(L^1_x \cap L^\infty_x\right)}.\nonumber
\end{gathered}
\end{equation}
Since
\begin{equation}
\begin{gathered}
\int_0^t \int_s^t e^{-\nu(p)(t-s_1)} \nu(p)(1+s_1 -s)^{-\sigma_r} (1+s)^{-\sigma_r}\,ds_1\,ds \lesssim (1+t)^{-\sigma_r},\nonumber
\end{gathered}
\end{equation}
as in (5.9), we have
\begin{equation}
\begin{gathered}
\left|\left| E_2 \right|\right|_{L^2_x \cap L^\infty_x}+\left|\left| E_3 \right|\right|_{L^2_x \cap L^\infty_x} \leq C_M \eta (1+t)^{-\sigma_r} \left|\left| \varpi_r h \right|\right|_{L^\infty_{t,p} \left(L^2_x \cap L^\infty_x \right)}.\nonumber
\end{gathered}
\end{equation}
Lastly, applying Theorem 5.1 to the first term of (5.6) and choosing $\eta$ small enough, we obtain the desired result.  
\end{proof}

{\it{E-mail address}}: nishimura.koya.42e@kyoto-u.jp  

\end{document}